\documentclass[smallcondensed]{svjour3}       

\usepackage{longtable}
\usepackage{graphicx}
\usepackage{caption,subcaption}
\usepackage{framed}
\usepackage{amssymb}
\usepackage{epstopdf}
\usepackage{amsmath,amsfonts,amssymb,mathtools}
\usepackage{hyperref}
\usepackage{cite}
\usepackage{enumitem}
\usepackage{booktabs}

\usepackage[boxruled]{algorithm2e}

\newtheorem{fact}[theorem]{Fact}

\def\RR{{\mathbb{R}}}
\def\va{\mathbf{a}}
\def\vb{\mathbf{b}}
\def\vg{\mathbf{g}}
\def\vh{\mathbf{h}}
\def\NN{{\mathbb{N}}}

\def\mcF{{\mathcal{F}}}

\def\dd{\delta}

\newcommand{\argmin}{{\rm arg}\!\min}

\newcommand{\co}{\mathrm{conv}}

\begin{document}

\title{An algorithm for best generalised rational approximation of continuous functions}

\author{R. D\'iaz Mill\'an  \and Nadezda Sukhorukova \and Julien Ugon}

\institute{
R. D\'iaz Mill\'an \at School of Information Technology, Deakin University, Geelong, Australia, E-mail: \href{mailto://r.diazmillan@deakin.edu.au}{r.diazmillan@deakin.edu.au}.
\and
Nadezda Sukhorukova \at Swinburne University of Technology, Hawthorn, Australia, E-mail: \href{mailto://nsukhorukova@swin.edu.au}{nsukhorukova@swin.edu.au}. 
\and
Julien Ugon \at School of Information Technology, Deakin University, Geelong, Australia, E-mail: \href{mailto://julien.ugon@deakin.edu.au}{julien.ugon@deakin.edu.au}.
}

\date{}


\maketitle
\begin{abstract}
    The motivation of  this paper is the development of an optimisation method for solving optimisation problems appearing in Chebyshev rational and generalised rational approximation problems, where the approximations are constructed as ratios of linear forms (linear combinations of basis functions). The coefficients of the linear forms are subject to optimisation and the basis functions are continuous function. It is known that the objective functions in generalised rational approximation problems are quasi-convex. In this paper we also prove a stronger result, the objective functions are pseudo-convex in the sense of Penot and Quang. Then we develop numerical methods, that are efficient for a wide range of pseudo-convex functions and test them on generalised rational approximation problems.  
\end{abstract}

\section{Introduction}
Consider the set of all real-valued polynomials with degree up to $n$, denoted by~$\Pi_n$, and the continuous function $f:\RR\rightarrow \RR$. We are interested in the problem of approximating the function $f$ by a rational function $\frac{p}{q}$ where $p\in \Pi_n$ and $q\in \Pi_m$ for some given nonnegative numbers $n\geq0$ and $m\geq 0$. In others words we want to solve the optimisation problem:
\begin{equation}\label{problempol}
 \min_{p\in\Pi_n, q\in\Pi_m} \sup_{t\in I}\left|f(t)-\frac{p(t)}{q(t)}\right|.
\end{equation}
Problem~(\ref{problempol}) is also known as Chebyshev rational approximation problem. 
 We represent the polynomial $p\in\Pi_n$ by an element in $\RR^{n+1}$ by 
 $$p(t)=\langle \va,\textbf{t}_n\rangle =a_0+a_1t+a_2t^2\cdots + a_nt^n,$$ where $\va=(a_0,a_1,a_2,\ldots,a_n)\in \RR^{n+1}$ are the coefficient of the polynomial, and the vector the vector $\textbf{t}_n=(1,t,t^2,\ldots,t^n)\in\RR^{n+1}$ for each $t\in I$, $I\subset \RR$ being any compact subset of $\RR$.
 Now Problem \eqref{problempol} can be written as:
 \begin{equation}\label{problem1pol}
  \min_{(\va,\vb)\in C} \Psi^f_{\Pi}(\va,\vb),
 \end{equation}
where $\Psi^f_{\Pi}(\va,\vb)=\sup_{t\in I}\left|f(t)-\frac{\langle\va,\textbf{ t}_{n}\rangle}{\langle\vb,\textbf{ t}_{m}\rangle}\right|$ is a maximal deviation for the approximation of $f$ on the interval $I$, and 
$$C=\{(\va,\vb)\in \RR^{n+1}\times \RR^{m+1}: \langle \vb, \textbf{ t}_m \rangle \geq 1, \forall t\in I\}$$ is a feasible set.


Rational approximation was a popular research topic in 50s-60s of the twentieth century~\cite{Achieser1965, CheneyLoeb,Ralston1965Reme,Rivlin1962} as a promising alternative to the free knot spline approximation. Rational approximation models combine simplicity and significant flexibility, two properties attractive for practical applications~\cite{Blairnumintegration, AMCPeirisSukhSharonUgon}.

In~\cite{cheney1964generalized} Cheney and Loeb demonstrated that some of the results in the area of Chebyshev rational approximation can be extended to approximation by a ratio of linear forms
\begin{equation}\label{eq:ratiolinearforms}
    \frac{G(\va,t)}{H(\vb,t)}=\frac{a_0g_0(t)+a_1g_1(t)+a_2g_2(t)\cdots + a_ng_n(t)}{b_0h_0(t)+b_1h_1(t)+b_2h_2(t)\cdots + b_mh_m(t)}=\frac{\langle \va,\textbf{ g}(t)\rangle}{\langle \vb,\textbf{ h}(t)\rangle},
\end{equation}
where $g_i(t)$, $i=0,\dots,n$ and $h_j(t), j=0,\dots, m$ are not limited to monomials. The authors call this type of approximation \textit{generalised rational approximation}. There are a number of ways to generalise rational approximations. In the current paper we use the same terminology as in~\cite{cheney1964generalized} and therefore we approximate continuous functions by the ratios of linear forms and the coefficients of these forms are subject to optimisation. It is still required for the linear form in the denominator to be positive. 

These extensions of the results are possible due to the fact that the corresponding objective functions in the optimisation problems are quasi-convex. We will talk about this property in Section~\ref{sec:preliminary}. In this paper we also prove a stronger result, the objective functions are pseudo-convex in the sense of Penot and Quang (which extends the notion of pseudo-convexity to the case of nonsmooth functions, see~\cite{aussel,aris,PenotQuang}). Then we develop numerical methods, that are efficient for a wide range of pseudo-convex functions and test them on generalised rational approximation problems.  

In the case of generalised rational approximation, the optimisation problem:
\begin{equation}\label{problem}
 \min_{\va,\vb} \sup_{t\in I}\left|f(t)- \frac{\langle \va,\textbf{ g}(t)\rangle}{\langle \vb,\textbf{ h}(t)\rangle}\right|,
\end{equation}
 where $\va=(a_0,a_1,a_2,\ldots,a_n)\in \RR^{n+1}$ and  $\vb=(b_0,b_1,\ldots,b_m)\in\RR^m$
 are the decision variables and also the coefficients of the linear forms, $g_i(t)$, $i=0,\dots,n$ and $h_j(t)$, $j=0,\dots,m$ are called the basis functions, $t\in I$, $I\subset \RR$ any compact subset of $\RR$.
 This can be written as:
 \begin{equation}\label{problem1}
  \min_{(\va,\vb)\in C} \Psi^f(\va,\vb),
 \end{equation}
where $\Psi^f(\va,\vb)=\sup_{t\in I}\left|f(t)- \frac{\langle \va,\textbf{ g}(t)\rangle}{\langle \vb,\textbf{ h}(t)\rangle}\right|$ is a maximal deviation for the approximation of $f$ on the interval $I$, and $$C=\{(\va,\vb)\in \RR^{n+1}\times \RR^{m+1}: \langle \vb,\textbf{ h}(t)\rangle \geq 1, \forall t\in I\}$$ is the feasible set.

The manuscript is organized as follows. In Section~2 we introduce notation and preliminaries. In Section~3 we introduce essential results related to the approximation problem. Section~4 is dedicated to developing an algorithm for solving a variational inequality which is equivalent to solve the approximation problem. In Section~5 we present a general algorithm to solve non-monotone variational inequalities for point-to-set operators. Section~6 shows some numerical experiments to demonstrate the behaviuor of the presented algorithm. Some conclusion remarks are presented in Section~7. 

\section{Preliminary results}\label{sec:preliminary}

This section is devoted to 
some classic notation, definition and results we will use along to the present work.

By $\RR^n$ we denote the $n$ dimensional Euclidean space, $|\cdot|$ is the absolute value function, $\|\cdot\|$ the norm induced by the inner product $\langle \cdot,\cdot \rangle$. The set $\co D$ is the convex hull of the set $D$. The orthogonal projection of a point $x\in\RR^n$, onto the convex, closed and non-empty set $C\subset \RR^n$, is defined by the unique point in $C$, which is the solution of the minimal distance problem:
 $$P_C(x)=\argmin_{y\in C}\|x-y\|.$$
Given a point-to-set the operator $T:\RR^n\rightrightarrows \RR^n$, the graph of $T$, is denoted and defined by $Gr(T):=\{(x,u)\in \RR^n\times\RR^n: u\in T(x)\}$. 
\begin{fact}\label{proj}
Let $C\subseteq \RR^n$ be closed, convex and non-empty. For all $x,y\in \RR^n$ and all $z\in C $, the following holds:
\begin{enumerate}
\item\label{proj-i} $\|P_C(x)-P_C(y)\|^2 \leq \|x-y\|^2-\|(x-P_C(x))-(y-P_C(y))\|^2.$
\item\label{proj-ii} $\langle x-P_C(x),z-P_C(x)\rangle \leq 0.$
\end{enumerate}
\end{fact}
\begin{proof}
See \cite[Proposition 4.8 and Theorem 3.14]{librobauch}.
\end{proof}
One of the most useful tools in projection algorithms is the \emph{Fej\'er convergence}, defined by:
 \begin{definition}[Fej\'er convergence]
 Let $S$ be a nonempty subset of $\RR^n$. A sequence $(x^k)_{k\in\NN}\subset \RR^n$ is said to be Fej\'er convergent to $S$, if and only if, for all $x\in S$ there exists $k_0\in \NN$ such that $\|x^{k+1}-x\| \le \|x^k - x\|$ for all $k\ge k_0$.
 \end{definition} 
  The main properties that a Fej\'er convergent sequences satisfy are: 
  
\begin{proposition}\label{fejer}
If $(x^k)_{k\in \NN}$ is Fej\'er convergent to $S$, then it is bounded.
\end{proposition}
\proof
By definition, taking $\bar{x}\in S$, we have $\|x^k-\bar{x}\|\leq \|x^n-\bar{x}\|$ for any $n\leq k$, then for all $k\in\NN$, $\|x^k-\bar{x}\|\leq \|x^{k_0}-\bar{x}\|$.
\endproof
 The concepts of monotonicity, pseudo-monotonicity and quasi-monotonicity are mentioned in this work. Now, we define these operators. 
 \begin{definition}
 A point-to-set operator $T:\RR^n\rightrightarrows \RR^n$ is called:
\begin{itemize}
\item[(i)] {\it Monotone}, if and only if, for all $(x,u), (y,v)\in Gr(T)$,
$$\langle u-v,x-y\rangle\geq 0.$$
\item [(ii)] {\it Pseudo-monotone}, if and only if, for all $(x,u), (y,v)\in Gr(T)$, the following implication holds:
$$\langle u,y-x\rangle\geq 0 \Longrightarrow \langle v, y-x\rangle\geq 0. $$
\item [(iii)] {\it Quasi-monotone}, if and only if, for all $(x,u), (y,v)\in Gr(T)$, the following implication holds:
$$\langle u,y-x\rangle > 0 \Longrightarrow \langle v, y-x\rangle\geq 0. $$
\end{itemize} 
 \end{definition}
 It is clear that every monotone operator is pseudo-monotone, and every pseudo-monotone operator is quasi-monotone.
\section{Approximation of continuous functions}
 This section is dedicated to analysing Problem~\eqref{problem1pol}. 

\begin{lemma}\label{quasicon}
For any real function $f:\RR\rightarrow \RR$ and a compact set $I\subseteq \RR$, the maximal deviation $\Psi^f:\RR^{n+1}\times\RR^{m+1}\rightarrow \RR$ is a quasi-convex function.
\end{lemma}
\begin{proof}
For any $t\in I$ the function 
$f(t)-\frac{\langle \va,\vg(t)\rangle}{\langle \vb,\vh(t)\rangle}$, with
$$\begin{aligned}\vg(t)&=\left(g_0(t),g_1(t),\dots,g_n(t)\right),\\\vh(t)&=\left(h_0(t),h_1(t),\dots,h_m(t)\right),\end{aligned}$$ 
${\langle \vb,\vh(t)\rangle}>0$
is a quasi-convex function, as the sum of a constant and the ratio of two linear functions~\cite{Boyd}. Then the function $\Psi^f$ is a quasi-convex function as a supremum of quasi-convex functions.
\end{proof}

Interestingly, this result is a direct corollary of \cite[Lemma~2]{cheney1964generalized}, but it was not elaborated by the authors and was left unnoticed for several decades. The most reasonable explanation for this is that most efficient techniques for quasi-convex optimisation were developed much later: around 20~years after~\cite{cheney1964generalized}.  

We denote by $A^+(\va,\vb)$ and $A^-(\va,\vb)$ the sets of actives values, i.e., 
$$
\begin{aligned}
A^+(\va,\vb) &=\left\{t\in I: \Psi^f(\va,\vb)=\sigma^f_t(\va,\vb)\right\}\\ A^-(\va,\vb) &=\left\{t\in I: \Psi^f(\va,\vb)=-\sigma^f_t(\va,\vb)\right\},\end{aligned}$$ where $\sigma^f_t(\va,\vb):=f(t)-\frac{\langle \va,\vg(t)\rangle}{\langle \vb,\vh(t)\rangle}$.

\begin{theorem}\label{clarke}
The Clarke subdifferential of the function $\Psi^f$ can be computed as follows:

\begin{equation}\label{subdifferential}
\partial \Psi^f(\va,\vb)=\co \left\{ \nabla \sigma^f_t(\va,\vb), -\nabla \sigma^f_l(\va,\vb): t\in A^+(\va,\vb), l\in A^-(\va,\vb) \right\}.
\end{equation}
\end{theorem}
\begin{proof}
See Theorem 10.31 in \cite{rocke-wets}.
\end{proof}

In~\cite{bura-mill}, the authors introduce an algorithm for solving variational inequalities, when the operator is pseudo-convex, subject to some continuity requirements. 
In the following example we show that the subdifferential of the function $\Psi^f$ is not necessarily inner semi-continuous, and therefore does not satisfy the requirements from~\cite{bura-mill}.

\begin{example}

Consider the constant function $f:\mathbb{R}\rightarrow \mathbb{R}$ defined by $f(t)=0$, the polynomials $p(t)=2t^2-1$ and $q(t)=1$, and the compact set $I=[-1,1]$. Then $\va=(-1,0,2)$ $\vb=(1)$, $A^+ = \{0\}$ and $A^-=\{-1,1\}$. Therefore $\partial \Psi^f(\va,\vb)=\co \{(1,-1,1,-2),(-1,0,0,1),(1,1,1,-2)\}$ corresponding to the active points $t_0=-1$, $t_1=0$ and $t_2=1$. Now consider the sequence $(\va_n,\vb_n)=(-1,1/n,2,1)$ which converges to the point $(\va,\vb)=(-1,0,2,1)$. For all $n\in\mathbb{N}$, $A^+({\bf a}_n,{\bf b}_n) = \emptyset$ and $A^-({\bf a}_n,{\bf b}_n) = \{1\}$. That is:
  $$\Psi^f(\va_n,\vb_n)=\sup_{t\in [-1,1]}\left|-1 + \frac{t}{n} + 2t^2\right|=1+\frac{1}{n}.$$
  Then for all $n\in\mathbb{N}$ we have that $\partial \Psi^f(\va_n,\vb_n)=(1,1,1,-1-\frac{1}{n})$, which proves that there is no sequence of elements of $\partial \Psi^f(\va_n,\vb_n)$ converging to $y=(-1,0,0,1)\in \partial \Psi^f(\va,\vb)$. Therefore the operator $\partial \Psi^f$ is not {\it inner semi-continuous}. 
\end{example}
The following Lemma is a direct consequence of the continuity of the function $\Psi^f$. 
\begin{lemma}\label{continu}
Consider the converging sequence $(\va_n,\vb_n)_{n\in\NN}\subset C$, so that  $$\lim_{n\to \infty}(\va_n,\vb_n)=(\va,\vb)\in C.$$ Taking for all $n\in \NN$ an active value $t_n\in A^+(\va_n,\vb_n)\cup A^-(\va_n,\vb_n)$, for any converging subsequence $(t_{k_n})_{k_n\in \NN}\subseteq (t_n)_{n\in\NN}$, such that $\lim_{k_n\to \infty}t_{k_n}=\bar{t}$ we have that $\bar{t}\in A^+(\va,\vb)\cup A^-(\va,\vb)$.  
\end{lemma}
\begin{proof}
The existence of a convergent subsequence is due to the boundedness of the compact set $I$, from now on, suppose for simplicity that the sequence $(t_i)_{i\in\NN}$ is convergent to a point $\bar{t}\in I$.

By continuity of $f$ we have
\begin{equation}\label{cont}
\lim_{i\to\infty}\left|f(t_i)-\frac{\langle {\bf a}_i,\vg(t_i)\rangle}{\langle {\bf b}_i,\vh(t_i)\rangle}\right|=\left|f(\bar{t})-\frac{\langle {\bf a},{\vg(\bar{t})}\rangle}{\langle {\bf b},{\vh(\bar{t})}\rangle}\right|.
\end{equation}

Take any $\hat{t}\in I$. For all $i\in\NN$ we have
$$
\left|f(\hat{t})-\frac{\langle {\bf a}_i,{\vg(\hat{t})}\rangle}{\langle {\bf b}_i,{\vh(\hat{t})}\rangle}\right|\leq  \left|f(t_i)-\frac{\langle {\bf a}_i,\vg(t_i)\rangle}{\langle {\bf b}_i,\vh(t_i)\rangle}\right|,
$$
because all $t_i$ are active values of $\Psi^f$ on its respective $({\bf a}_i,{\bf b}_i)$. Taking limits when $i\to \infty$ we obtain, using the continuity of $\Psi^f$, that
$$\left|f(\hat{t})-\frac{\langle {\bf a},{\vg(\hat{t})}\rangle}{\langle {\bf b},{\vh(\hat{t})}\rangle}\right|\leq\left|f(\bar{t})-\frac{\langle {\bf a},{\vg(\bar{t})}\rangle}{\langle {\bf b},{\vh(\bar{t})}\rangle}\right|.$$
 Therefore $\bar{t}$ is an active value. 
\end{proof}
\begin{proposition}\label{optimal}
Given a continuous function $f$, if $0\in \partial \Psi^f(\va,\vb)$ then $(\va,\vb)$ is a global minimiser of $\Psi^f$.
\end{proposition}
\begin{proof}
  This proof is largely based on \cite{cheney1964generalized}.

  Define by 
  $$s_t(\va,\vb) = \frac{\sigma^f_t(\va,\vb)}{\Psi^f(\va,\vb)},$$ so that $s_t(\va,\vb) = 1$ (resp. $-1$) when  $t\in A^-(\va,\vb)$ (resp. $t\in A^+(\va,\vb)$).

  Then, $\partial\Psi^f = \co\{s(\va,\vb)\nabla\sigma^f_t(\va,\vb), t\in A^+(\va,\vb)\cup A^-(\va,\vb)\}$, where

  \begin{align*}
\nabla\sigma^f_t(\va,\vb) &=\begin{multlined}[t] \frac{1}{\langle \vb,\vh(t)\rangle^2} \big(-g_0(t)\langle\vb,\vh(t)\rangle,-g_1(t)\langle\vb,\vh(t)\rangle,\ldots,-g_n(t)\langle\vb,\vh(t)\rangle,\\
h_0(t)\langle\va,\vg(t)\rangle,h_1(t)\langle\va,\vg(t)\rangle,\ldots,h_m(t)\langle\va,\vg(t)\rangle\big)
\end{multlined}\\
 &= \frac{1}{\langle \vb,\vh(t)\rangle^2}(\langle \vb,\vh(t)\rangle \vg(t),\langle\va,\vg(t)\rangle \vh(t)).
 \end{align*}

  Suppose that the point $(\va,\vb)$ is \emph{not} the global optimiser. In this case, there is a direction $(\va',\vb')\neq 0_{n+m+2}$ such that $\Psi^f(\va-\va',\vb-\vb')<\Psi^f(\va,\vb)$, and $(\va-\va',\vb-\vb')\in C$. In particular for any $t\in A^+(\va,\vb)\cup A^-(\va,\vb)$, $|\sigma^f_t(\va-\va',\vb-\vb')|<|\sigma^f_t(\va,\vb)|$.

  From there, we can find that:
  \begin{align*}
    s_t(\va,\vb) \bigg(f - \frac{\langle \va,\vg(t)\rangle}{\langle \vb,\vh(t)\rangle}\Bigg) - s_t(\va,\vb) \bigg( f - \frac{\langle \va-\va',\vg(t)\rangle}{\langle \vb-\vb',\vh(t)\rangle}\bigg) &> 0 \\
    s_t(\va,\vb) \bigg( \frac{\langle \va-\va',\vg(t)\rangle}{\langle \vb-\vb',\vh(t)\rangle}- \frac{\langle \va,\vg(t)\rangle}{\langle \vb,\vh(t)\rangle}\bigg) &>0\\
    s_t(\va,\vb) \bigg( \langle \va-\va',\vg(t)\rangle- \frac{\langle \va,\vg(t)\rangle}{\langle \vb,\vh(t)\rangle}\langle \vb-\vb',\vh(t)\rangle\bigg) &>0\\
    s_t(\va,\vb) ( -\langle \va',\vg(t)\rangle + \frac{\langle \va,\vg(t)\rangle}{\langle \vb,\vh(t)\rangle}{\langle \vb',\vh(t)\rangle}) &>0\\
    \frac{s_t(\va,\vb)}{\langle \vb,\vh(t)\rangle} \Big( -\langle \vb,\vh(t)\rangle\langle \va',\vg(t)\rangle + \langle \va,\vg(t)\rangle{\langle \vb',\vh(t)\rangle}\Big) &>0\\
    \frac{s_t(\va,\vb)}{\langle \vb,\vh(t)\rangle^2} \Big( \langle \va',-\langle \vb,\vh(t)\rangle\vg(t)\rangle + {\langle \vb',\langle \va,\vg(t)\rangle\vh(t)\rangle}\Big) &>0\\
    \Big\langle (\va',\vb') , s_t(\va,\vb)\nabla \sigma^f_t(\va,\vb) \Big\rangle &>0.
  \end{align*}
  In the above steps, multiplying and dividing the left hand sides by $\langle \vb-\vb',\vh(t)\rangle$ and $\langle \vb,\vh(t)\rangle$ does not change the sign of the inequalities, since both of these quantities are positive.

  Since this product is positive for every $t\in A^+\cup A^-$, it is also positive for every element in $\partial^f\Psi(\va,\vb)$. This implies that $0\notin \partial^f\Psi(\va,\vb)$.
\end{proof}
\begin{theorem}
  The function $\Psi^f$ is a pseudo-convex function in terms of~\cite{aussel,aris,PenotQuang}. Consequently the Clarke Subdifferential $\partial \Psi^f$ is a pseudomonotone operator.
\end{theorem}
\begin{proof}
Following from Lemma \ref{quasicon}, Proposition \ref{optimal} of this paper and Theorem~4.1 in \cite{aussel}, we have that $\Psi^f$ is pseudo-convex. The second part is due to Proposition 2.2 in \cite{aris}.
\end{proof}
This result is especially important, since the basis functions are not restricted to monomials.
\section{Variational inequality and the algorithm}

In this section we develop an algorithm for solving the variation Inequality Problem defined below. With this purpose we present a version of \textbf{Algorithm F} in \cite{be-mi-pa}. 

Consider the operator $T:\RR^n\rightrightarrows \RR^n$ and the set $C\in\RR^n$, the variational inequality problem for $T$ and $C$, denoted by VIP(T,C), is defined as:
\begin{equation}\label{vip}
\text{Find } x^*\in C \text{ : } \exists u^*\in T(x^*) \text{ : } \langle u^*,x-x^*\rangle\geq 0, \forall x\in C.
\end{equation}
When the operator $T$ is a pseudomonotone operator, Problem~\eqref{vip} is equivalent to the dual variational inequality problem (DVIP(T,C)):

\begin{equation}\label{dualvip}
\text{Find } x^*\in C \text{ : } \forall u\in T(x) \text{ : } \langle u,x-x^*\rangle\geq 0, \forall x\in C.
\end{equation}

 We denote the solution of Problem 
 \eqref{vip} by $S^*$ and the solution of Problem \eqref{dualvip} by $S_0$. The equivalence of the Problems \eqref{vip} and \eqref{dualvip} results in $S^*=S_0$.
 
 Now, we present the version of the Linesearch F suitable to our problem: 
 
 \SetKwFor{ForAll}{for all}{}{end}

 \begin{algorithm}[H]
 \DontPrintSemicolon
 \KwIn{$(a,b)\in C$, $\beta>0$ and $\dd\in(0,1)$}
 Set $\alpha\leftarrow 1$ and $\theta \in (0,1)$. \;
 \ForAll{$(u_a,u_b)\in  \left\{ \nabla \sigma^f_t(a,b), -\nabla \sigma^f_l(a,b): t\in A^+(a,b), l\in A^-(a,b)\right\}$}
 {
 define $(z_a,z_b)=P_C((a,b)-\beta (u_a,u_b))$\;
 \uIf{for each $(u_a,u_b)$ we have $$\max_{(u^{\alpha}_{a},u^{\alpha}_b)\in D_{\alpha}}\langle (u^{\alpha}_{a},u^{\alpha}_b),(a,b)-(z_{a},z_{b})\rangle < \delta \langle (u_a,u_b), (a,b)-(z_a,z_b)\rangle,$$
 where $D_{\alpha}:= \left\{ \nabla \sigma^f_t(a_{\alpha},b_{\alpha}), -\nabla \sigma^f_l(a_{\alpha},b_{\alpha}): t\in A^+(a_{\alpha},b_{\alpha}), l\in A^-(a_{\alpha},b_{\alpha})\right\}$ and \\ $(a_{\alpha},b_{\alpha})= \alpha(z_a,z_b)+(1-\alpha)(a,b),$
 }{$\alpha\leftarrow\theta \alpha$}
 \Else{Return $\alpha$ and $(u_a,u_b)$}
 }
 \KwOut{$(\alpha,(u_a,u_b))$}
 \caption{\textbf{LineSearch F:} feasible direction}\label{feasible}
 \end{algorithm}







\begin{algorithm}[H]\DontPrintSemicolon
\KwIn{$(\beta_k)_{k\in \NN}\subset[\check{\beta},\hat{\beta}]$ such that $0<\check{\beta}\le \hat{\beta}<+\infty$ and $\dd\in(0,1)$.}
\textbf{Initialization:} Take $(a^0,b^0)\in C$ and set $k\leftarrow 0$.\;

\textbf{Step~1: } Apply Algorithm~\ref{feasible} to compute
\[
\left(\alpha_k,(u_a^k,u_b^k)\right)= \textbf{ Linesearch  F}\big((a^k,b^k),\beta_k,\dd\big),
\]\;
 Set $(z_a^k,z_b^k)=P_C\left((a^k,b^k)-\beta_k (u_a^k,u_b^k)\right)$ then we have
\[
\left\langle (u_a^{\alpha_k},u_b^{\alpha_k}),(a^k,b^k)-(z_a^k,z_b^k)\rangle\geq\delta\langle (u_a^k,u_b^k),(a^k,b^k)-(z_a^k,z_b^k)\right\rangle
\]
with $(\bar{a}^k,\bar{b}^k)=\alpha_k (z_a^k,z_b^k)+(1-\alpha_k)(a^k,b^k)$ and $(u_a^{\alpha_k},u_b^{\alpha_k})\in \partial\Psi^f\left(\bar{a}^k,\bar{b}^k\right)$.\;
\textbf{Step~2 (Stopping Criterion):} 
\lIf{$(z_a^k,z_b^k)=(a^k,b^k)$ or $(a^k,b^k)=P_C\left((a^k,b^k)-(v_a^k,v_b^k)\right)$ with $(v_a^k,v_b^k)\in \partial\Psi^f((a^k,b^k))$,}{stop}
\textbf{Step 3:} Set
\begin{subequations}
\begin{align}
(\bar{a}^k,\bar{b}^k)&:=\alpha_k (z_a^k,z_b^k)+(1-\alpha_k)(a^k,b^k),{\label{xbar2}}\\
\text{and}\quad
(a^{k+1},b^{k+1})&:=\mcF\left((a^k,b^k)\right);\label{P1121}
\end{align}
\end{subequations}
\textbf{Step~4:} If $(a^{k+1},b^{k+1})=(a^k,b^k)$, then stop. Otherwise, set $k\leftarrow k+1$ and go to \textbf{ Step~1}.\;
\caption{\textbf{Algorithm F}}\label{A1}
\end{algorithm}

\noindent We consider three variants of this algorithm. Their main difference lies in way to compute \eqref{P1121}:
\begin{align}
\mcF_{\rm 1}\left((a^k,b^k)\right) =& P_C\left(P_{H\left((\bar{a}^k,\bar{b}^k),(u_a^{\alpha_k},u_b^{\alpha_k})\right)}\left((a^k,b^k)\right)\right);\label{P112}  \quad   &{(\bf Variant\; 1)} \\
\mcF_{\rm 2}\left((a^k,b^k)\right) =& P_{C\cap H\left((\bar{a}^k,\bar{b}^k),(u_a^{\alpha_k},u_b^{\alpha_k})\right)}\left((a^k,b^k)\right);\label{P122}  \quad   &{(\bf Variant\; 2)}
\end{align} where
\begin{equation}\label{H(x,v)}
H(x,u):=\left\{ y\in \RR^{n+m+2} : \langle u,y-x\rangle\le 0\right\}. 
\end{equation}

\subsection{Convergence}

This section is dedicated to prove the convergence of \textbf{ Algorithm F}. We first show that \textbf{Linesearch F} terminates. 
\begin{proposition}\label{linewell}
 If $(a,b)\in C$ is not a solution of Problem \eqref{vip}, \textbf{Linesearch F} terminates after finitely many iterations. 
\end{proposition}
\begin{proof}
Suppose that \textbf{Linesearch F} never stops. Then, for all $\alpha \in\{1,\theta,\theta^2,\cdots\}$ and $(u_a,u_b)\in \partial \Psi^f(a,b)$ we have 
$$\left\langle (u_a^{\alpha},u_b^{\alpha}),(a,b)-(z_a,z_b)\rangle<\delta\langle (u_a,u_b),(a,b)-(z_a,z_b)\right\rangle $$
for all $(u_a^{\alpha},u_b^{\alpha})\in \partial \Psi^f\left(a_\alpha,b_\alpha\right)$, with $(a_\alpha,b_\alpha)=\alpha  (z_a ,z_b )+(1-\alpha )(a,b)$.  Note that by the structure of the set $\partial \Psi^f(a,b)$ it is enough to verify the inequality over the points in $D_\alpha$. Taking limits when $\alpha\to 0$, we have that $$(a_\alpha,b_\alpha)=\alpha(z_a,z_b)+(1-\alpha)(a,b)\to (a,b).$$ 
 By Lemma \ref{continu}, there exists $(u_a,u_b)\in \partial \Psi^f(a,b)$ such that
 
\begin{equation}\label{imply}
\left\langle (u_a,u_b),(a,b)-(z_a,z_b)\rangle\leq\delta\langle (u_a,u_b),(a,b)-(z_a,z_b)\right\rangle. 
\end{equation} 
  Which implies that $(1-\delta)\langle (u_a,u_b),(a,b)-(z_a,z_b)\rangle\leq 0$ since $\delta\in(0,1)$ we have
  \begin{align*}
  0&\geq \langle (u_a,u_b),(a,b)-(z_a,z_b)\rangle\\
  &\geq\|(a,b)-(z_a,z_b)\|^2+\langle (z_a,z_b)-\left((a,b)-(u_a,u_b)\right),(a,b)-(u_a,u_b)\rangle.
  \end{align*}
   Using now Fact~\ref{proj} we have that $$\langle (z_a,z_b)-\left((a,b)-(u_a,u_b)\right),(a,b)-(u_a,u_b)\rangle\geq 0,$$ then 
   $$ 0\geq \langle (u_a,u_b),(a,b)-(z_a,z_b)\rangle\geq\|(a,b)-(z_a,z_b)\|^2\geq 0,$$
which implies that $\|(a,b)-(z_a,z_b)\|^2=0$, then $(a,b)=(z_a,z_b)$ implying that $(a,b)$ is a solution of Problem \ref{vip}, which is a contradiction.
\end{proof}
\begin{proposition}\label{xkinhk}
$(a^k,b^k)\in S^*$ if and only if $(a^k,b^k)\in H\left((\bar{a}^k,\bar{b}^k),(u_a^{\alpha_k},u_b^{\alpha_k})\right)$.
\end{proposition}
\begin{proof}
See Proposition 4.3 in \cite{be-mi-pa}.
\end{proof}
\begin{proposition}
If \textbf{ Algorithm F} stops at steps 2 or 4, then it stops at the solution. 
\end{proposition}
\begin{proof}
See Proposition 4.4 in \cite{be-mi-pa}.
\end{proof}
\subsubsection{Convergence Analysis of $\mcF_{\rm 1}$}

Now we provide the convergence analysis of the \textbf{ Variant $\mcF_{\rm 1}$}. From now on, we assume that \textbf{Algorithm F} produces an infinite sequence $(a^k,b^k)_{k\in\NN}\nsubseteq S_*$.
Due to Proposition 4.5(i) in \cite{be-mi-pa} we have the Fej\'er convergence and consequently, the boundedness, by Proposition \ref{fejer}, of the sequence $(a^k,b^k)_{k\in\NN}$. Also 

\begin{equation}
\lim_{k\to\infty}\langle (u_a^{\alpha_k},u_b^{\alpha_k}),(a^k,b^k)-(\bar{a}^k,\bar{b}^k)\rangle=0.\label{limite}
\end{equation} 

The next theorem will be proved following the ideas of the Theorem 4.6 in \cite{be-mi-pa}, adapted to our point-to-set case. 

\begin{theorem}\label{xkconverf1}
The sequence $(a^k,b^k)_{k\in\NN}$ converges to a point in $S_*$.
\end{theorem}
\begin{proof}
For the sake of brevity, we will not reproduce the full proof given in \cite{be-mi-pa}. Instead we will specify whenever the proofs are different due to the operator, in our case being point-to-set, and point out that we consider all elements on the normal cone of $C$ as zero. 

Consider a subsequence $(i_k)_{k\in\NN}$ such that all sequences involved in the algorithm be convergent, i.e., $(a^{i_k},b^{i_k})\to (\hat{a},\hat{b})$, $(u_a^{\alpha_{i_k}},u_b^{\alpha_{i_k}})\to (\tilde{u}_a,\tilde{u}_b)$ and $(u_a^{i_k},u_b^{i_k})\to (\hat{u}_a,\hat{u}_b)$. It is possible because all sequences are bounded.\\
 Using \eqref{limite} and the same criteria used in \cite{be-mi-pa}, we have that
 
\begin{equation}
\lim_{k\to\infty}\alpha_{i_k}\|(a^{i_k},b^{i_k})-(z_a^{i_k},z_b^{i_k})\|=0,\label{gotozero}
\end{equation}
 and here as in~\cite{be-mi-pa}, we consider two cases:  $\alpha_{i_k}\to \hat{\alpha}> 0$ or $\alpha_{i_k}\to 0$. \\
 
\textbf{ Case 1:} $\lim_{k\in\NN}\alpha_{i_k}=\hat{\alpha}>0$. Due to Lemma \ref{continu} we have that, $$\lim_{k\to\infty}(u_a^{\alpha_{i_k}},u_b^{\alpha_{i_k}})=(\tilde{u}_a,\tilde{u}_b)\in \partial\Psi^f(\hat{a},\hat{b}).$$ By the continuity of the projection mapping and \eqref{gotozero}, $(\hat{a},\hat{b})=P_C\left((\hat{a},\hat{b})-(\tilde{u,}_a\tilde{u}_b)\right)$, which implies that $(\hat{a},\hat{b})\in S_*$. 

\textbf{ Case 2:} $\lim_{k\to\infty}\alpha_{i_k}=0$. Defining $\tilde{\alpha}_{i_k}=\frac{\alpha_{i_k}}{\theta}$, then $\tilde{\alpha}_{i_k}\to 0$, let $(\tilde{a}^{i_k},\tilde{b}^{i_k})=\tilde{\alpha}_{i_k}(z_a^{i_k},z_b^{i_k})+(1-\tilde{\alpha}_{i_k})(a^{i_k},b^{i_k})$ then $(\tilde{a}_{i_k},\tilde{b}^{i_k})\to (\hat{a},\hat{b})$. Due to the \textbf{ Linesearch F}, we have that for all $(\tilde{v}_a^{i_k},\tilde{v}_b^{i_k})\in\partial\Psi^f(\tilde{a}^{i_k},\tilde{b}^{i_k})$ and all $(u_a^{i_k},u_b^{i_k})\in\partial\Psi^f(a^{i_k},b^{i_k})$ and $k\in\NN$ we have

\begin{equation}\label{nopass}
\langle(\tilde{v}_a^{i_k},\tilde{v}_b^{i_k}),(a^{i_k},b^{i_k})-(z_a^{i_k},z_b^{i_k})\rangle<\delta\langle (u_a^{i_k},u_b^{i_k}), (a^{i_k},b^{i_k})-(z_a^{i_k},z_b^{i_k})\rangle.
\end{equation}

Since $\tilde{\alpha}_{i_k}$ converges to zero and following the same idea as in the proof of Lemma~\ref{continu}, there exist sequences $(\tilde{\tilde{v}}_{a}^{i_k},\tilde{\tilde{v}}_{b}^{i_k})\in \partial \Psi^f(\tilde{a}^{i_k},\tilde{b}^{i_k})$ and $(\tilde{u}_a^{i_k},\tilde{u}_b^{i_k})\in \partial\Psi^f(a^{i_k},b^{i_k})$ such that $\lim_{k\to \infty}\|(\tilde{\tilde{v}}_{a}^{i_k},\tilde{\tilde{v}}_{b}^{i_k})-(\tilde{u}_a^{i_k},\tilde{u}_b^{i_k})\|=0$. Taking a convergent subsequence, we have $(\tilde{\tilde{v}}_{a}^{i_k},\tilde{\tilde{v}}_{b}^{i_k})\to (\tilde{u}_a,\tilde{u}_b)$ and $(\tilde{u}_a^{i_k},\tilde{u}_b^{i_k})\to (\tilde{u}_a,\tilde{u}_b)$ as well, also $(a^{i_k},b^{i_k})\to (\hat{a},\hat{b})$ and $(z_a^{i_k},z_b^{i_k})\to (\hat{z}_a,\hat{b}_b)=P_{C}\left((\hat{a},\hat{b})-(\tilde{u}_a,\tilde{u}_b)\right)$. Then passing to the limits on equation \eqref{nopass} we obtain
$$
\langle (\tilde{u}_a,\tilde{u}_b), (\hat{a},\hat{b})-(\hat{z}_a,\hat{z}_b)\rangle\leq \delta \langle (\tilde{u}_a,\tilde{u}_b), (\hat{a},\hat{b})-(\hat{z}_a,\hat{z}_b)\rangle.
$$

The above equation is the same as in \eqref{imply}, the proof continuous as in Proposition \ref{linewell}, getting that the limits of $(a^{i_k},b^{i_k})$ is a solution of the problem.

Now as the sequence is Fej\'er convergent to the solution set, following Proposition 4.5 (i) in \cite{be-mi-pa}, we get the convergence to the solution set of the whole sequence. 

\end{proof}
\subsubsection{Analysis convergence of $\mcF_{\rm 2}$}

For the case of $\mcF_{\rm 2}$ all the proof are the same as was proved in \cite{be-mi-pa}. 
\section{General Algorithm for non-monotone Variational Inequality}
The {\bf Algorithm F} and the {\bf Linesearch F} can be applied to more general problem. In fact, consider $T:\RR^n\rightrightarrows\RR^n$, and the convex and closed set $C\subseteq \RR^n$. With the following conditions, the {\bf Algorithm G}({\bf G} for ``general'') below, is convergent. 
\begin{itemize}
\item [A1)] $T$ is closed.
\item [A2)] $T$ is bounded on bounded sets. 
\item[A3)] Problem \eqref{dualvip} and \eqref{vip} are equivalents. That is, $S^*=S_0$.
\end{itemize}
By closed we mean that the graph of $T$ is closed. A2 is a classical assumption. A3 is weaker than pseudo-monotone, see examples in \cite{bura-mill}. \\
 Following we present the general version for the {\bf Linesearch F} and {\bf Algorithm F}. \\
 
\begin{algorithm}[H]
 \DontPrintSemicolon
 \KwIn{$x\in C$, $\beta>0$ and $\dd\in(0,1)$}
  Set $\alpha\leftarrow 1$ and $\theta \in (0,1)$. \;
   \ForAll{ $u\in T(x)$, Define $z=P_C(x-\beta u)$.}
   {
 \uIf {for each $u\in T(x)$ \;
 $$\max_{u^{\alpha}\in T(x_\alpha)}\langle u^{\alpha},x-z\rangle < \delta \langle u, x-z\rangle,$$}

  where $x_{\alpha}:= \alpha z+(1-\alpha)x,$
 $\alpha\leftarrow\theta \alpha$\;
 \Else{Stop and choose $u^{\alpha}\in T(x_{\alpha})$ such that:
 $\langle u^{\alpha},x-z\rangle\geq\langle u,x-z\rangle$}
 
  }
 
 \KwOut{$(\alpha,u^{\alpha})$}
 \caption{\textbf{LineSearch G:} General lineasearch}\label{gene}
 \end{algorithm}

\begin{algorithm}[H]\DontPrintSemicolon
\KwIn{$(\beta_k)_{k\in \NN}\subset[\check{\beta},\hat{\beta}]$ such that $0<\check{\beta}\le \hat{\beta}<+\infty$ and $\dd\in(0,1)$.}
\textbf{Initialization:} Take $(x^0)\in C$ and set $k\leftarrow 0$.\;

\textbf{Step~1: } Apply Algorithm~\ref{gene} to compute
\[
\left(\alpha_k,u^{\alpha_k}\right)= \textbf{ Linesearch  G}\big(x^k,\beta_k,\dd\big),
\]\;
 Set $z^k=P_C\left(x^k-\beta_k u^k)\right)$ then we have
\[
\left\langle u^{\alpha_k},x^k-z^k \rangle\geq\delta\langle u^k,x^k-z^k\right\rangle
\]
with $u^{\alpha_k} \in T(\bar{x}^k)$ and $\bar{x}^k=\alpha_k z^k+(1-\alpha_k)x^k$ .\;
\textbf{Step~2 (Stopping Criterion):} 
\lIf{$z^k=x^k$ or $x^k=P_C\left(x^k-v^k\right)$ with $v^k\in T(x^k)$,}{stop}
\textbf{Step 3:} Set
\begin{subequations}
\begin{align}
\bar{x}^k&:=\alpha_k z^k+(1-\alpha_k)x^k,\\
\text{and}\quad
x^{k+1}&:=\mcF\left(x^k\right);\label{fxk}
\end{align}
\end{subequations}
\textbf{Step~4:} If $x^{k+1}=x^k$, then stop. Otherwise, set $k\leftarrow k+1$ and go to \textbf{ Step~1}.\;
\caption{\textbf{Algorithm G}}\label{AG}
\end{algorithm}

Two variants can be consider for computing \eqref{fxk}.
\begin{align*}
\mcF_{\rm 1}(x^k) =& P_C\left(P_{H\left(\bar{x}^k,u^{\alpha_k}\right)}(x^k)\right); \, \,  \\
\mcF_{\rm 2}(x^k)=& P_{C\cap H(\bar{x}^k,u^{\alpha_k})}(x^k);\, \,  
\end{align*} 
where $u^{\alpha_k}\in T(\bar{x}^k)$ and 

\begin{equation*}
H(x,u):=\left\{ y\in \RR^{n} : \langle u,y-x\rangle\le 0\right\}.
\end{equation*}

The differences between the {\bf Algorithm G} and the {\bf Conceptual Algorithm F} in \cite{be-mi-pa} are listed bellow:
\begin{itemize}
\item We consider null normal vectors. 
\item We consider point-to-set operators instead of point-to-point. 
\item The {\bf Linesearch G} is different to {\bf Linesearch F}.
 \end{itemize}
  
\section{Numerical Experiment}

This section is dedicated to showing some numerical experiments to demonstrate the development of the \textbf{ Algorithm F} to approximate continuous functions. The section is split into two parts. The first part is dedicated to the "classical approximation" (using rational functions with polynomials) we split this section onto two, first, we approximate function for which we don't know the solution sets and the compact set $I=[-1,1]$, and after, we approximate function when we know the solution sets, which allow to measure the distance to the solution set, and considering the compact set $I$ an equidistant points onto the interval $[-1,1]$.   . In the second part, we uses non-polynomials rational functions and a finite collection of equidistant points onto the interval $[-1,1]$ as the compact set $I$. We compare different kind of approximation rational functions.  

\subsection{Rational Approximation with Polynomial}
In this subsection we consider the Problem \eqref{problem} with $p\in\Pi_n$ and $q\in \Pi_m$, polynomials of degree $n$ and $m$ respectively. 
\subsubsection{Unknown solution over an interval}

In these examples, we choose the best results after 200 iterations. In all cases, the interval $I=[-1,1]$, $\Psi^f_i$ with $i=1,2$ denotes the function value for the Variants 1 and 2 respectively, {\it iter} is the iteration number in which was attained the best result. $n, m$ are the degree of the numerator and denominator polynomials respectively. In the pictures below, we specified the function and the degrees in each cases, in all pictures the blue color is for the Variant 1, the red color for Variant 2. The green color in figures 2,3,4 and 5, is the graph of the function $f$ which is the function to be approximated.

\begin{table}[ht]
\begin{tabular}{|cc|cc|cc|}  \hline 
 & \multicolumn{3}{c}{\textbf{ Algorithm F} for unknown solution} & & \\    \hline  
$f(t)$ & $(n,m)$ & $Iter$ & $\Psi^f_1(a,b)$ & $Iter$& $\Psi^f_2(a,b)$  \\   \hline
$|t|$ & $(2,2)$                                      & 161 & 0.068 &186 & 0.070  \\
$|t|$  & $(3,3)$                                           & 193  & 0.072 & 191 & 0.069  \\
$|t|$ & $(4,3)$  &198 &0.069 &198 & 0.069\\ \hline
$\sin(t)$ & $(2,2)$                                        & 186  & 0.0097  & 186 &0.014   \\
$|\sin(t)|$                                          & $(3,3)$& 189  & 0.0714  & 193& 0.0700  \\ 
$\sqrt{|t|}$ & $(4,4)$  & 193 &0.186 & 185 & 0.182   \\  \hline
\end{tabular}
\end{table}
  \parbox{\linewidth}{\centering
     \includegraphics[width=1.2\textwidth]{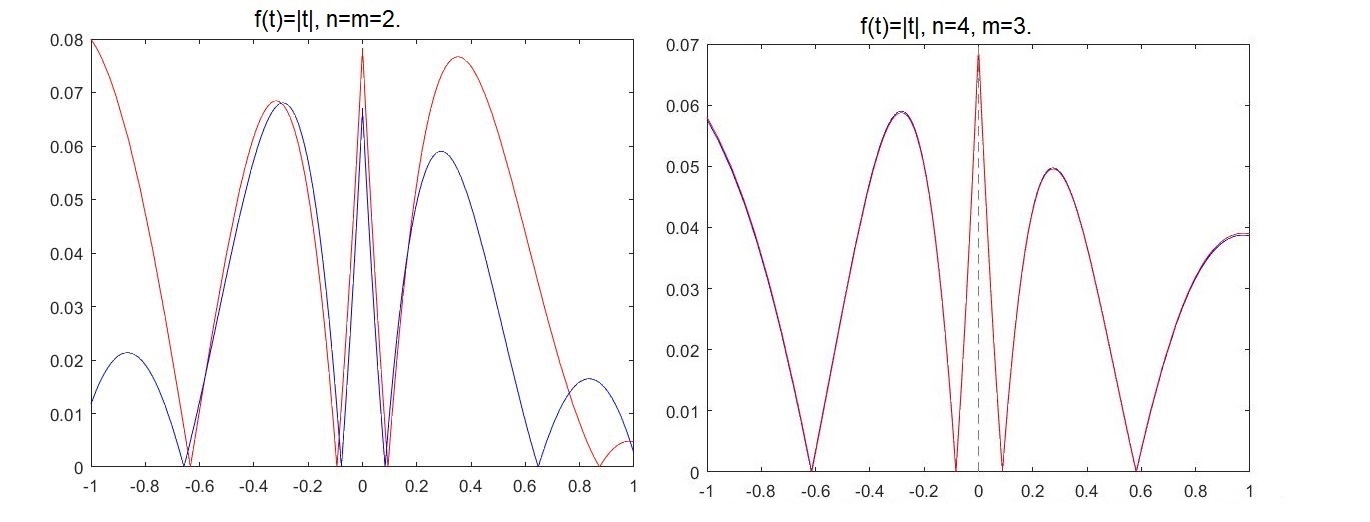} 
   \\
        \captionof{figure}{In both sides the objective function $\Psi_i^f$, $i=1,2$.}}
    
\parbox{\linewidth}{\centering
\includegraphics[width=1.2\textwidth]{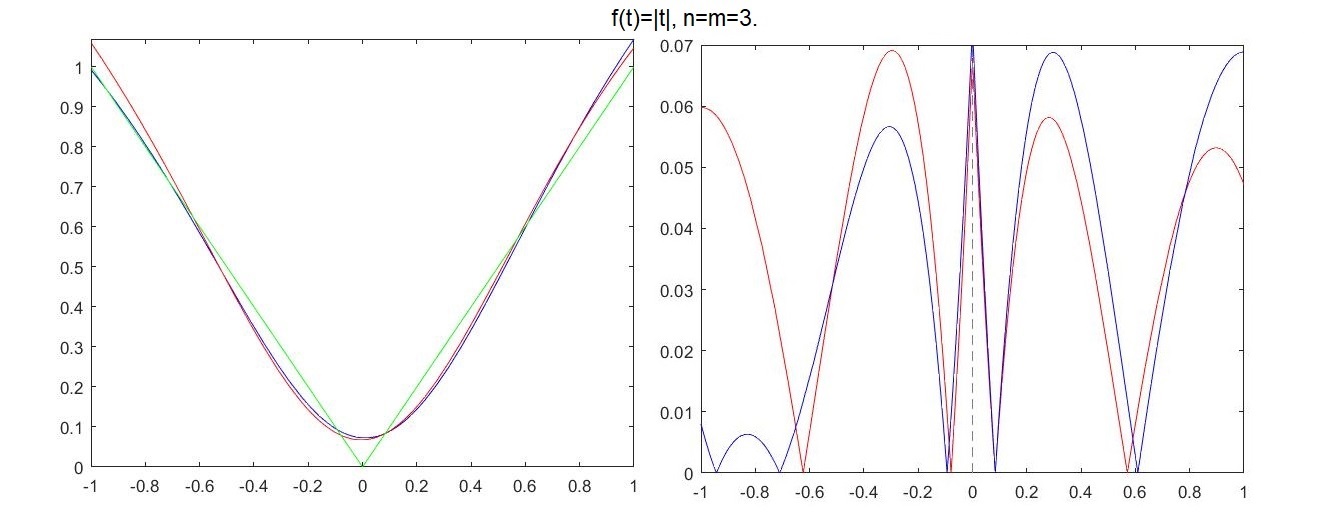} 

\captionof{figure}{\textbf{ Left:} The graph of $f$ and $p/q$. \textbf{ Right:} Objective function $\Psi^f_i$, $i=1,2$.}}

\parbox{\linewidth}{\centering
\includegraphics[width=1.2\textwidth]{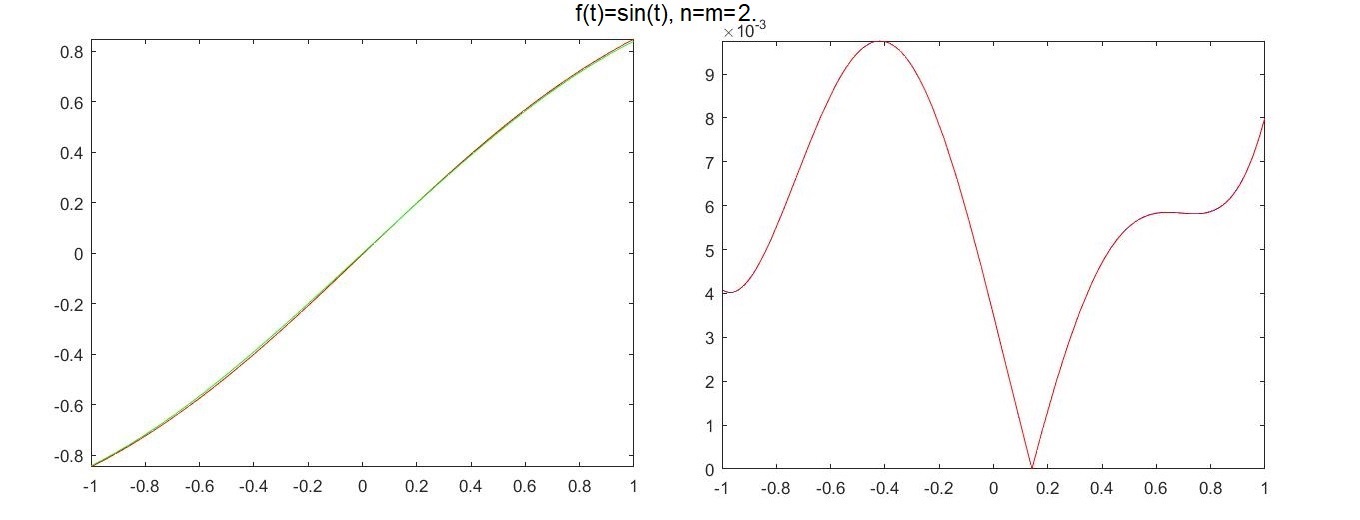} 
\captionof{figure}{\textbf{ Left:} The graph of $f$ and $p/q$. \textbf{ Right:} Objective function $\Psi^f_i$, $i=1,2$.}}

\parbox{\linewidth}{\centering
\includegraphics[width=1.2\textwidth]{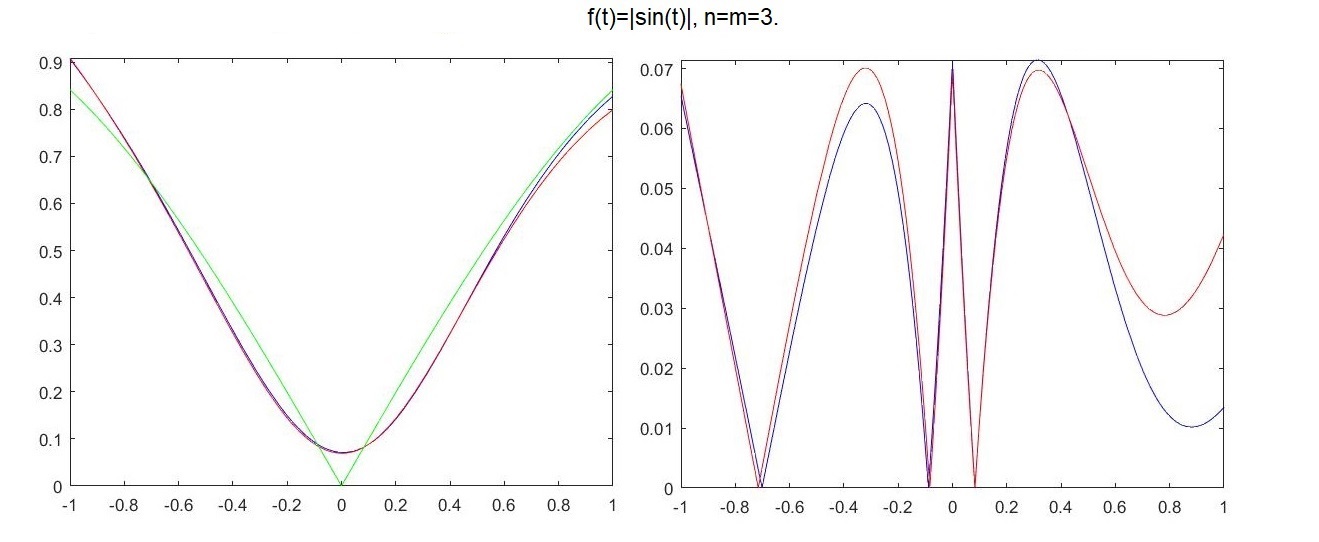} 

\captionof{figure}{\textbf{ Left:} The graph of $f$ and $p/q$. \textbf{ Right:} Objective function $\Psi^f_i$, $i=1,2$.}}

\parbox{\linewidth}{\centering
\includegraphics[width=1.2\textwidth]{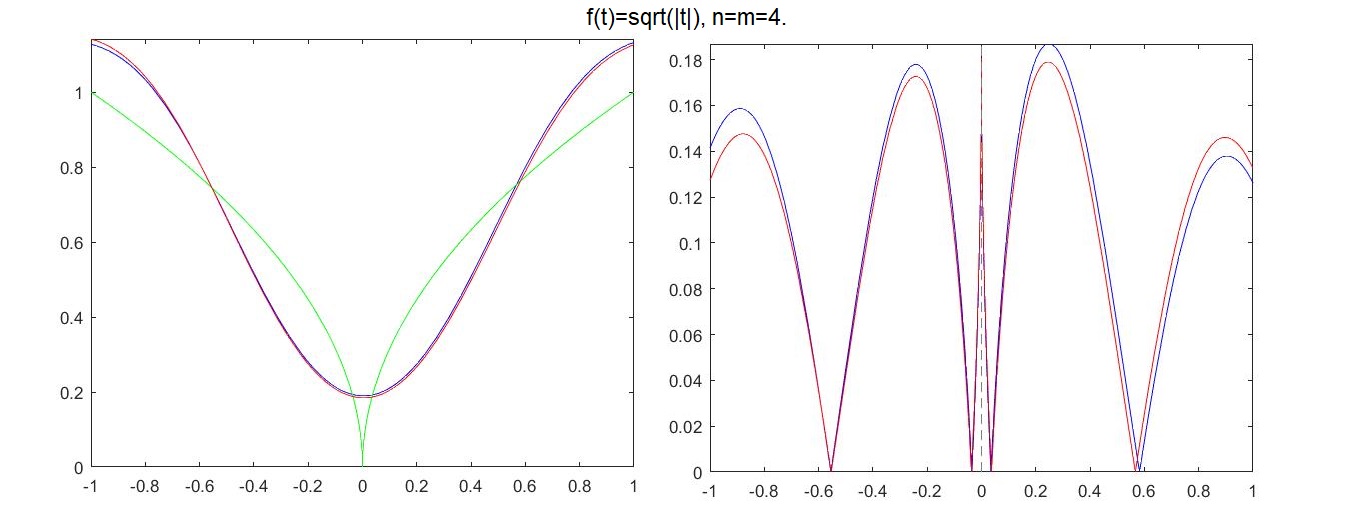} 
\captionof{figure}{\textbf{ Left:} The graph of $f$ and $p/q$. \textbf{ Right:} Objective function $\Psi^f_i$, $i=1,2$.}}

\subsubsection{Known solution over discrete compact set}

In this subsection we testing rational functions as the objective function. Then we know the solution sets. We denote by {\it iter} the number of iterations, $d_i(x,S^*)$ denotes the distance between the last point to the solution set for the variant $i$. We stopped the algorithm when the function value $\Psi^f_i$ at the current point be less or equal to $10^{-3}$.  For the compact set $I$, we used a collection of $M$ equidistant points on the interval $[-1,1]$. 

\begin{table}[ht]
\begin{tabular}{|ccc|cc|cc|}  \hline 
 & \multicolumn{3}{c}{\textbf{ Algorithm F} for known solution}    \\   \hline  
$f(t)$ & $(n,m)$ & M & $Iter$ & $d_1(x,S^*)$ & $Iter$& $d_2(x,S^*)$  \\   \hline
$1$ & $(1,1)$& 100  & 12 & $9.746*10^{-4}$ &12 & $9.87*10^{-4}$ \\ 
$1$  & $(2,2)$& 200 & 455 &0.004 & 381 & 0.0026  \\ \hline
$\frac{1}{t^2+1}$ & $(1,2)$ &100 &423 &0.0097 &135 & 0.0042\\ 
$\frac{1}{t^2+1}$ & $(2,2)$ &200 &4406 &0.0977 &3950 & 0.0972\\ \hline
$\frac{t}{t+1.5}$ & $(1,1)$& 100 & 4240 & 0.013  & 3137 &0.0069   \\ 
$\frac{t}{t+1.5}$ & $(2,2)$& 200 & 6490  &  0.0172  & 5875 &0.009   \\ \hline
$\frac{t^2-1}{t+2}$ & $(2,2)$ & 200 &2361  & 0.02 & 1730& 0.0437  \\ 
 $\frac{t^2-1}{t+2}$ & $(3,2)$ & 100 &14643  & 0.0901 & 6306& 0.0396  \\ \hline
\end{tabular}
\end{table}

\subsection{Non-polynomial rational approximation}
In this section we consider different rational function to approximate continuous functions. We denote by {\it CPU} the CPU time. By $h(t)$, we denote the functions which compose the rational function, i.e., $p(t)=\langle{\bf a},{\bf H}(t)\rangle=\sum_{k=0}^{n}a_k h^k(t)$. For this subsection we consider as the objective function the continuous function $f(t)=\frac{\sin t-\cos t}{t+2}$. The compact set in all cases are the collection of $M$ equidistant points in the interval $[-1,1]$. The used stopping criteria is based on the value of the objective function $\Psi_i$, with $i=1,2$ for the two variants. We used $\Psi_i({\bf a},{\bf b})\leq 10^{-2}$. The propose of this examples is compare both variants in the number of iterations and the CPU time. 
\begin{table}[ht]
\begin{tabular}{|ccc|cc|cc|}  \hline 
 & \multicolumn{3}{c}{\textbf{ Algorithm F} with non-polynomial}    \\   \hline  
$h (t)$ & $(n,m)$ & M & $Iter_1$ & $CPU_1$ & $Iter_2$& $CPU_2$  \\   \hline
$e^t$ & $(3,3)$& 20  & 625 & $16.9063$ &410 & $11.3594$ \\ 
$e^t$  & $(3,3)$& 100 & 585 &21.1875 & 569 & 23.4063 \\ 
$e^t$ & $(5,4)$ &200 &2115 &420.531 &2094 &419.281\\ 
$e^t$ & $(10,8)$ &100 &5121 &997.875 &5099 & 940.172\\ \hline
$\sin(t)$ & $(3,3)$& 20 & 255& 12.0313 & 255 &11.4513   \\ 
$\sin(t)$ & $(3,3)$& 100 &429  &  44.6094  & 311 & 35.3906   \\ 
$\sin(t)$ & $(5,4)$ & 200 & 235  &52.1094 & 184& 37.4688  \\ 
 $\sin(t)$ & $(10,8)$ & 100 &136  & 23.9531 & 214& 20.9688  \\ \hline
\end{tabular}
\end{table}

\section{Conclusion}
 In this paper a new Algorithm, containing two variants, for point-to-set operators without continuity and monotonicity is proposed. We proved that the Clarke subdifferential of the maximal deviation function is a pseudomonotone operator, which is a better than the known quasimonotonicity. An application to solve Rational Approximation problems are tested with different rational functions, and the algorithm obtain a good behaviour and comparing both Variants, the best results in general was obtained by the {\bf Variant 2}, which project onto a small set and just once. 

In this paper we also prove that the corresponding objective function in the case of generalised rational approximation are pseudo-convex and therefore it opens a broad avenue for further investigating. In particular, we can apply general methods developed for pseudo-convex functions to generalised rational approximation. Another potential research direction is to investigate the objective functions and optimisation problems in the case of multivariate approximation.

\section*{Acknowledgement}
This research was supported by the Australian Research Council (ARC),  Solving hard Chebyshev approximation problems through nonsmooth analysis (Discovery Project DP180100602).

\bibliographystyle{plain}

\end{document}